\documentclass [11pt]{amsart}
\usepackage {amsmath, amssymb, amscd, amsthm, mathrsfs, url ,pinlabel,hyperref, verbatim, lipsum}
\usepackage[text={5.5in,9.1in},centering,letterpaper,dvips]{geometry}
\usepackage{color,multirow,dcpic,latexsym,pictexwd,graphics,pdfpages}
\usepackage{fancyhdr, tikz, setspace}
\usepackage[all,cmtip]{xy}
\usepackage{mathtools}
\usepackage{graphicx} 
\usepackage{subcaption}
\usetikzlibrary{positioning, arrows.meta}
\usepackage{scrextend}

\usepackage[T1]{fontenc}
\usepackage{lmodern}
\DeclareUrlCommand{\bfurl}{}

\newtheorem {theorem}{Theorem}
\newtheorem {lemma}[theorem]{Lemma}
\newtheorem {definition}[theorem]{Definition}
\newtheorem {proposition}[theorem]{Proposition}
\newtheorem {corollary}[theorem]{Corollary}

\theoremstyle{remark}

\newtheorem {remark}[theorem]{Remark}

\numberwithin{equation}{section}
\numberwithin{theorem}{section}

\title{On nests and large components of random real algebraic curves}
\author{Turgay Bayraktar}
\thanks{This work is partially supported by T\"{U}B\.{I}TAK 1001 project 124F150} 
\address{Faculty of Engineering and Natural Sciences, Sabanc{\i} University, Tuzla, Istanbul, T\"{u}rkiye}
\email{\href{mailto: tbayraktar@sabanciuniv.edu.tr}{tbayraktar@sabanciuniv.edu.tr}}
\author{Ali Ula\c{s} \"{O}zg\"ur K{\.i}\c{s}{\.i}sel}
\address{Department of Mathematics, Middle East Technical University, Ankara, 06800, T\"{u}rkiye}
\email{\href{mailto: akisisel@metu.edu.tr}{akisisel@metu.edu.tr}}

\date{}

\begin{document}

\begin{abstract}
We develop a variant of the barrier method in order to address questions about topology of Kostlan random real algebraic plane curves. In particular we prove that the expected number of connected components of the curve of length 
at least $\displaystyle{O\left(\sqrt{d^{-1}\log \log d}\right)}$ grows to infinity with $d$, and likewise, the expected number of nests of the curve of depth at least $\displaystyle{O\left(\log\log d\right)}$ grows to infinity with $d$. In another direction, we adapt an $L^{\infty}$-norm bound result of Shifmann and Zelditch to subspaces and employ it to obtain a lower bound for the probability that a finite number of points remain all in different components of the complement of a large degree random curve. 
\end{abstract}
\maketitle

\section{Introduction} 
Random real algebraic geometry is flourishing in several directions as a mathematical subject. One of the important avenues is to understand the topology of random real subvarieties, about which certain essential results were established in \cite{FLL}, \cite{GW1}, \cite{GW2}  and other works in the literature. Of particular interest is on finding lower bounds on the expected number of connected components, or more generally on finding lower bounds on the expected Betti numbers of the real part of the algebraic variety. Most of the lower bound estimates discovered for answering these questions so far use the ``barrier method'', which was developed in \cite{NS} for random nodal sets and then used in \cite{GW2} for Kostlan random real algebraic sets. This method comprises of starting from a reference polynomial whose real zero locus has a desired topological property and then perturbing the reference polynomial so that these perturbations retain the aforementioned property. In order for the properties in question to be stable under sufficiently large perturbations so as to give meaningful lower bounds on expectation, often a reference polynomial with certain extremal properties is used to begin with. Typically, these are polynomials constructed by using Hörmander peak sections or by using the Bergman kernel, so that their $L^{\infty}$-norms are large compared to their $L^2$-norms and their pointwise norms are concentrated in the neighborhood of interest.   

In this paper, we develop a variant of the barrier method and use it to address certain questions about large components and nests of Kostlan random real algebraic plane curves, as well as a question about the probability of finitely many points in $\mathbb{R}P^2$ all to remain in different connected components of the complement of a random curve. The curves are randomized with respect to the unitary invariant Kostlan distribution which is induced in turn by the Fubini-Study metric. The precise definitions are recalled in Section \ref{preliminaries}, which contains the preliminaries needed for the rest of the paper.

By results of \cite{GW2}, in every ball of Fubini-Study radius $\displaystyle{d^{-1/2}}$, any fixed configuration of ovals occurs in a random curve  with positive probability independent of $d$, for $d$ sufficiently large. By using this fact, one can show that the expected number of components of length at least $\displaystyle{O\left(d^{-1/2}\right)}$ is bounded below by $cd$ for some positive constant $c$ independent of $d$, and in fact this is optimal in terms of the asymptotic order, because of the upper bounds proven in \cite{GW1}. There remains the question of presence of larger components than order $\displaystyle{d^{-1/2}}$. We give a partial answer to this question in Theorem \ref{largecomponents}, which states that for any $\epsilon>0$, the expected number of connected components of length at least $\displaystyle{O\left(\sqrt{d^{-1}\log(\epsilon \log d)}\right)}$ is bounded below by $\displaystyle{ \frac{cd^{1-\epsilon}}{\log(\log d)}}$, where $c$ is a positive constant independent of $d$. In particular, (provided that $\epsilon$ is small) this expectation goes to infinity as $d$ grows. 

A second important question about the topology of a real algebraic smooth plane curve is on the presence of nests of ovals. Any oval of a curve in $\mathbb{R}P^2$ has a well-defined interior which is diffeomorphic to a disk, and an exterior which is diffeomorphic to a M\"obius band. A finite collection of ovals is said to form a nest if they can be ordered such that each oval in the sequence contains all the preivious ones in its interior. The number of ovals in such a nest is called its depth. There is an easy deterministic upper bound on the depth of any nest in any degree $d$ smooth plane curve, which is $\displaystyle{d/2}$. This can be proven as an immediate application of Bezout's theorem by intersecting the curve with a line passing through the interior of the innermost oval in the nest. For Kostlan random real algebraic curves, Diatta and Lerario in \cite{DL} proved, by low-degree approximations of polynomials, that the very deep nests are exponentially rare, more precisely the probability of finding a nest of order $O(d)$ exponentially decays to zero (and so does the expected number of such nests). In the lower bound direction, the aforementioned result in \cite{GW2} about the presence of any fixed configuration of ovals in a neighborhood of radius $\displaystyle{d^{-1/2}}$ with positive probability can be used to deduce that, given any finite number $M$, the expected number of nests of depth at least $M$ grows at least (and in fact, also, by \cite{GW1}, at most) as $O(d)$. 
In this paper, we give an improvement of this result in the sense that the expected number of nests with depth growing to infinity with $d$ also grows to infinity with $d$. More precisely, in Theorem \ref{expecteddepth} we prove that for any $\epsilon>0$ the expected number of nests of depth at least $O(\log(\epsilon\log d))$ in a Kostlan random real plane curve is bounded below by $\displaystyle{\frac{cd^{1-\epsilon}}{\log(\log d)}}$ for some positive number $c$ independent of $d$. 

In order to prove these results on the topology of random real curves, we develop a variant of the barrier method, differing from the one in \cite{GW2} in several aspects. First of all, instead of controlling the $C^{1}$-norm of the polynomials in a given region as in \cite{GW2}, we control the sup-norm only on the boundary of the region. This relaxation has the price that the zero locus of the perturbed polynomial does not have to be isotopic to the zero locus of the reference polynomial, however, we make use $\mathbb{Z}_{2}$-homology invariants for which the control of the sup-norm on the boundary happens to suffice - this is explained in subsection \ref{annulus}. The sup-norm bounds on the random perturbation polynomials are obtained by using the mean-value inequality as in \cite{NS} or \cite{GW2}; we use the formulation in \cite{KW}. In subsection \ref{annulus}, the mean-value inequality is slightly adapted to a neighborhood of scale $\displaystyle{\sqrt{d^{-1}g(d)}}$, where $1\ll g(d)\ll d$, so that it can be used in applications such as Theorems \ref{largecomponents} and   \ref{expecteddepth}. For the applications on nests and Theorem \ref{expecteddepth}, the selection of the reference polynomial poses a delicate problem: Naively forming a reference polynomial by multiplying quadratic equations so as to give nested ovals results in a polynomial with a very poor control on its $L^2$ and sup-norms. Instead, by a careful use of Chebyshev polynomials of the first kind, we overcome this problem and obtain a reference polynomial with the desired properties. The definition and basic properties of Chebyshev polynomials of the first kind are recalled in subsection \ref{subsectionChebyshev}, whereas the reference polynomial is constructed in subsection \ref{Chebyshevimprove}.  
 
A different approach for controlling the global $L^{\infty}$-norm of a random polynomial was given by Shifmann and Zelditch in \cite{SZ}, where it was proven that the ratio of the $L^{\infty}$-norm to the $L^2$-norm of a Kostlan random polynomial rarely exceeds $O(\sqrt{\log d})$. Their method works more generally on any K\"{a}hler manifold equipped with a positive hermitian line bundle. This sup-norm bound is not as sharp as what the mean-value inequality gives in small neighborhoods, however it has the novel feature of providing a global bound, hence it can be used for simultaneous perturbations at far-away points. As a showcase of how this ``global barrier method'' can be applied to a geometric problem, we prove in Theorem \ref{finitelymanypoints} that for any positive integer $m$, and $\epsilon>0$, and $\displaystyle{p_{1},\ldots,p_{m}\in \mathbb{R}P^2}$ with the Fubini-Study distances between any two of them greater than $2d^{-\frac{1}{2}+\epsilon}$, there exists $\beta(m)>0$ such that the probability of all $p_{i}$ to remain in different connected components of the complement of a Kostlan random real plane curve is at least $\displaystyle{d^{-\beta(m)}}$ for $d$ sufficiently large. The use of this global barrier method requires a technical enhancement of a theorem in \cite{SZ} of Shifmann and Zelditch so that it is adapted also to work in a proper subspace of the complete linear system. We state and prove this enhancement in Theorem \ref{lem:hyperplane-tail} for hyperplanes and in Theorem \ref{thm:alpha} for general subspaces, which might be results of interest on their own. It remains a question of interest whether for intermediate-sized neighborhoods one could benefit from interpolating between the mean-value inequality for small neighborhoods and Shifmann-Zelditch type global bounds for the whole manifold. 

\section{Preliminaries} \label{preliminaries} 

\subsection{The Fubini-Study metric on $\mathbb{C}P^2$} \label{FS}
Let $\mathbb{C}P^2$ denote the complex projective plane. The tautological line bundle 
$$\tau=\mathcal{O}_{\mathbb{C}P^{2}}(-1)=\{([\ell],v)| v\in \ell\}\subset \mathbb{C}P^2 \times \mathbb{C}^3$$is a rank 1 subbundle of the trivial rank $3$ bundle on $\mathbb{C}P^2$, hence it inherits a hermitian product from the standart one on $\mathbb{C}^3$, called the Fubini-Study hermitian product. This in turn canonically induces a Fubini-Study hermitian product $h_d$ on every line bundle $\mathcal{O}_{\mathbb{C}P^2}(d)\cong(\tau^{*})^{\otimes d}$ on $\mathbb{C}P^{2}$ for all integer values of $d$. 

We recall the expression of this hermitian product in local charts. Let $[X_{0}:X_{1}:X_{2}]$ denote homogeneous coordinates on $\mathbb{C}P^2$. Let $U_{0}$ be the affine chart defined as $U_{0}=\{[X_{0}:X_{1}:X_{2}]\in \mathbb{C}P^2 | X_{0}\neq 0 \}\cong\mathbb{C}^2$. We identify $z=(z_{1},z_{2})\in \mathbb{C}^2$ with $[1:z_{1}:z_{2}] \in \mathbb{C}P^2$. For every $d\geq 0$, the complex vector space $H_{d}$ of polynomials in two variables of degree at most $d$ in $z_{1}, z_{2}$ (or equivalently, homogeneous polynomials in $X_{0},X_{1},X_{2}$ of degree $d$) has dimension $\displaystyle{N_{d}={d+2 \choose 2}}$ and it can be canonically identified with $H^{0}(\mathbb{C}P^{2}, \mathcal{O}_{\mathbb{C}P^2}(d))$. Under this identification, the formula for the pointwise Fubini-Study norm in affine coordinates becomes 
\begin{equation} \label{pointwise} 
 \|P(z)\|^2_{FS} =     \frac{|P(1,z)|^2}{(1+\|z\|^2)^d}. 
 \end{equation}
See \cite{KW}, subsection 2.4 for a detailed discussion. 

The curvature form $\omega_{FS}$ of the Fubini-Study metric on the hyperplane bundle $L=\tau^{*}=\mathcal{O}(1)$ induces a unitary invariant Riemannian metric on $\mathbb{C}P^2$, hence it endows $\mathbb{C}P^2$ with the notions of Fubini-Study volume and Fubini-Study distance. Regarding these differential geometric notions, we note here the following well-known results.   

\begin{lemma} \label{FSdistance}
(a) For every $z\in \mathbb{C}^2 \cong U_{0}$, $\displaystyle{dVol_{FS|_{z}}=\frac{dVol}{(1+\|z\|^2)^3}}$. In particular, $\displaystyle{Vol_{FS}(\mathbb{C}P^2)=\pi^2/2}$

(b) For every $z\in \mathbb{C}^2$, $\tan(d_{FS}(0,z))=d(0,z)$ where $d_{FS}$ denotes Fubini-Study distance and $d$ denotes Euclidean distance. 
\end{lemma} 
For a proof of these statements, see \cite{KW}, Lemma 2.11 and Lemma 2.12. 

\subsection{Kostlan distribution} 

The pointwise Fubini-Study hermitian product defined in subsection \ref{FS} defines an $L^2$-Fubini-Study hermitian product defined on the complex vector space $H_{d}$ by the following formula. 
\begin{equation}\label{hermprod}
\langle P,Q\rangle_{2}=\frac{1}{Vol_{FS}(\mathbb{C}P^{2})}\int_{\mathbb{C}P^{2}} h_{d}(P(z),Q(z))_{FS} dVol_{FS}. 
\end{equation}  
This in turn endows $H_{d}$ with the following Gaussian probability measure, called the Kostlan distribution, 
\begin{equation*} \label{Kostlan} 
d\mu_{\mathbb{C}}(P)=\frac{e^{-\|P\|^2_{2}}}{\pi^{N_d}}dP, 
\end{equation*} 

\noindent where $dP$ denotes the Lebesgue measure. Writing each complex polynomial as the sum of a polynomial with real coefficients and another one with purely imaginary coefficients, one obtains an orthogonal direct sum decomposition 
\begin{equation}\label{realim}
H_{d}=Re(H_{d})\oplus iIm(H_{d}).
\end{equation}
The induced probability measures on the real and imaginary parts take the form 
\begin{equation}\label{realKostlan}
d\mu_{\mathbb{R}}(P)=d\mu_{i\mathbb{R}}(P)=\frac{e^{-\|P\|^2_{2}}}{\sqrt{\pi}^{N_d}}dP 
\end{equation} 

The monomials in $H_{d}$ turn out to provide an orthogonal basis for the $L^2$-Fubini-Study inner product. More precisely, the following lemma holds. (For a proof, see \cite{Kos}, Theorem 4.1.)
\begin{lemma} \label{L2norms} 
For every $\mathbf{i}=(i_{0},i_{1},i_{2})\in \mathbb{N}^{3}$,
$\displaystyle{ \|X^{\mathbf{i}}\|_{2}^2 = {d+2 \choose \mathbf{i}}^{-1}, }$
where $d=i_{0}+i_{1}+i_{2}$, $X^{\mathbf{i}}=X_{0}^{i_{0}}X_{1}^{i_{1}}X_{2}^{i_{2}}$ and 
\[ {d+2 \choose \mathbf{i}}= \frac{(d+2)!}{i_{0}!i_{1}! i_{2}! 2! }\cdot\]
Moreover, if $\mathbf{j}=(j_{0},j_{1}, j_{2})\neq \mathbf{i}$, then $\langle X^{\mathbf{i}},X^{\mathbf{j}}\rangle_{2}=0$.  \hfill $\Box$ 
\end{lemma}

\subsection{Chebyshev polynomials of the first kind} \label{subsectionChebyshev}
Let $n\in \mathbb{N}$. The $n$th Chebyshev polynomial of the first kind is defined via the trigonometric equation 
\[ T_{n}(\cos(\theta))= \cos(n\theta). \] 
\noindent This relation defines $T_{n}(x)$ uniquely as a polynomial of degree $n$. Chebyshev polynomials of the first kind will be used as a starting point to construct curves with nesting which grows to infinity with $d$ in subsection \ref{Chebyshevimprove}. The following two classical propositions about Chebyshev polynomials of the first kind will be stated below without proof. For their proofs we refer the reader to \cite{MH}. 

\begin{proposition} \label{recursion} 
$T_{0}(x)=1$, $T_{1}(x)=x$ and the polynomials $T_{n}(x)$ satisfy the following recursion relation for $n\geq 1$. 
\begin{equation}\label{recursionequation}
T_{n+1}(x)=2x T_{n}(x)-T_{n-1}(x). 
\end{equation}
\hfill $\Box$ 
\end{proposition}
\begin{proposition} \label{roots}
All roots of $T_{n}(x)$ are real and  they are contained in the interval $[-1,1]$. More precisely, these roots are 
\[ x_{k}=\cos\left(\frac{(2k+1)\pi}{2n}\right), \quad k\in \{0,1,\ldots,n-1\}.  \] 
The extremal values of $T_{n}(x)$ in the interval $[-1,1]$ are all $\pm 1$ and they are located at the points
\[ y_{k}=\cos\left(\frac{k\pi}{n}\right), \quad k\in \{0,1,\ldots, n\}. \]
\hfill $\Box$ 
\end{proposition}
Let us write the $n$th degree Chebyshev polynomial $T_{n}(x)$ explicitly as 
\[ T_{n}(x)=\sum_{j=0}^{n} a_{j,n} x^j. \]
The following estimate on the coefficients $a_{j,n}$ will be useful in subsection \ref{Chebyshevimprove}. Let us assume that $a_{j,n}=0$ for $n<0$ or for $j>n$ for convenience. 

\begin{lemma} \label{coefficient}
For all $n\in \mathbb{N}$ and $0\leq j\leq n$, 
\[  |a_{j,n}| \leq 2^j  {n \choose j}.  \] 
\end{lemma} 

\noindent \textit{Proof:} We proceed by induction on $n$. The statement holds for $n=0$. The recursion relation in Proposition \ref{recursion} gives 
\[ a_{j,n}=2a_{j-1,n-1}- a_{j,n-2} \]
\noindent upon comparing the coefficients of $x^{j}$ on both sides of equation \eqref{recursionequation}. Assume that the statement holds upto the values of the second index including $n-1$. Then, 
\begin{eqnarray*} 
|a_{j,n}|&\leq& 2|a_{j-1,n-1}|+|a_{j,n-2}| \\
&\leq&2\cdot 2^{j-1} {n-1 \choose j-1} + 2^{j} {n-2 \choose j}  \\ 
&\leq&  2^{j} \left( {n-1 \choose j-1} +{n-1 \choose j}\right)  \\ 
&=& 2^{j} {n \choose j} 
\end{eqnarray*} 
\hfill $\Box$   
\section{Sup-norm bounds for random sections}
\subsection{Mean-value inequality and local sup-norm bounds} \label{meanvalue}
Let $g:\mathbb{N^{+}}\rightarrow \mathbb{R}_{\geq 0}$ be a monotone increasing function such that $1 \ll g(d)\ll d$ for $d$ large. Set $\displaystyle{R(d):=\sqrt{\frac{g(d)}{d}}}$ and $\displaystyle{\rho(d)=\arctan\left(R(d)\right)}$. By Lemma \ref{FSdistance}, part (b), for any point $z\in U_{0}\equiv \mathbb{C}^{2}\subset \mathbb{C}P^2$, the Euclidean ball $\displaystyle{B\left(z,R(d)\right)}$ coincides with the Fubini-Study ball $B_{FS}(z,\rho(d))$. Following \cite{KW}, we define the {\it ball-average norm} of a random polynomial in this Fubini-Study ball. 

\begin{definition} \label{ballaverage}
(Compare with  Definition 4.6 and Subsection 4.4 of \cite{KW}) The ball-average norm over the Fubini-Study ball $B_{FS}(z,\rho(d))$ is the random variable $Y_{d}$ taking $Q\in H^{0}(\mathbb{CP}^2,L^d)\mapsto \|Q\|_{B_{FS}(z,\rho_{d})}\in \mathbb{R}_{+}$
where the Fubini-Study ball-average norm $\|Q\|_{B_{FS}(z,\rho(d))}$ is defined by
\[  \|Q\|^2_{B_{FS}(z,\rho(d))}=\frac{1}{Vol(B_{FS}(z,\rho(d)))} \int_{B_{FS}(z,\rho(d))} \|Q(t)\|_{FS}^2 dVol_{FS}(t). \]     
\end{definition} 

Let $K$ be a compact subset of the Euclidean ball $\displaystyle{B\left(0,R(d)/2\right)}$ and let $P\in H^{0}(\mathbb{CP}^2,L^d)$ be a (reference) polynomial of degree $d$, which does not vanish at any point of $K$. By using the mean-value inequality for holomorphic functions and following the argument of \cite{KW}, subsection 4.3, we now estimate the sup-norm of any polynomial $Q\in H^{0}(\mathbb{CP}^2,L^d)$ on $K$ in terms of its Fubini-Study ball-average norm. For all $z\in K$ and all $Q\in H^{0}(\mathbb{CP}^2,L^d)$, by the mean value inequality,
\begin{eqnarray*}
\frac{ \|Q(z)\|_{FS}^2}{\|P(z)\|_{FS}^2} &=& \frac{ |Q(z)|^2}{|P(z)|^2} \\
&\leq& \frac{1}{|P(z)|^2 Vol\left(B\left(z,R(d)/2\right)\right)} \int_{B\left(z,R(d)/2\right)}  |Q(t)|^2dVol(t) \\ 
&\leq& \frac{1}{\|P(z)\|_{FS}^2 Vol\left(B\left(z,R(d)/2\right)\right)} \int_{B\left(z,R(d)/2\right)}  \|Q(t)\|_{FS}^2 \left( \frac{1+\|t\|^2}{1+\|z\|^2}\right)^{d}dVol(t) \\ 
\end{eqnarray*}
We next use the inequalities $\displaystyle{ \|z\|\leq R(d)/2}$ and $\displaystyle{\|t\|-\|z\|\leq R(d)/2}$ to deduce that 
\begin{eqnarray*} 
 \left( \frac{1+\|t\|^2}{1+\|z\|^2}\right)^{d} &=& \left( 1+ \frac{(\|t\|-\|z\|)(\|t\|-\|z\|+2\|z\|)}{1+\|z\|^2}\right)^{d} \\
&\leq& \left( 1+\frac{3(R(d))^2}{4}\right)^{d}  \\ 
&=& \left( 1+\frac{3g(d)}{4d}\right)^{d}  \\  
&\leq& \exp\left(\frac{3g(d)}{4}\right). 
\end{eqnarray*} 

\noindent Since $z\in K$, we have $\displaystyle{B\left(z,R(d)/2\right)\subset B(0,R(d))}$. Noting that $\displaystyle{Vol(B(0,R(d)))=2^4 Vol\left(B\left(z,R(d)/2\right)\right)}$ and by using Lemma \ref{FSdistance}, part (a), we obtain
\[ \frac{\|Q(z)\|^2_{FS}}{\|P(z)\|^2_{FS}} \leq \frac{2^4  \exp\left(3g(d)/4\right)}{\|P(z)\|^2_{FS} Vol(B(0,R(d)))} \int_{B(0,R(d))} \|Q(t)\|^2_{FS}(1+\|t\|^2)^3 dVol_{FS}(t). \] 

\noindent Finally, by using the estimate $\displaystyle{ Vol(B(0,R(d))= \int_{B(0,R(d))} (1+\|t\|^2)^3 dVol_{FS}(t) \geq Vol(B_{FS}(0, \rho(d)))}$, we arrive at the following local sup-norm bound. 

\begin{proposition} \label{localsupbound}
\begin{equation}
 \sup_{z\in K} \left( \frac{\|Q(z)\|^2_{FS} }{\|P(z)\|^2_{FS}} \right) \leq  \frac{2^4  \exp\left(3g(d)/4\right) \left(1+g(d)/d\right)^3 }{\inf_{z\in K}  \|P(z)\|^2_{FS} }  \|Q\|^2_{B_{FS}(0,\rho(d))}.  
\end{equation} 
\hfill $\Box$
\end{proposition}

\subsection{Global sup-norm bounds of random sections in large subspaces}
Let $M$ be any compact complex algebraic manifold of dimension $m$,  $(L,h)$ a positive hermitian line bundle on $M$, and $H_{d}:=H^{0}(M,L^{d})$. The asymptotic form of Riemann-Roch theorem implies that $N_{d}= \dim_{\mathbb{C}}H_{d}\asymp d^{m}$, see for instance Theorem 1.1.24 in \cite{Laz}. The pointwise hermitian structure $h$ naturally induces an $L^2$-hermitian inner product $\langle \cdot,\cdot \rangle$ and norm $\|\cdot\|$ on $H_{d}$. Let
$W_d \subset H_d$ be a complex linear subspace of dimension $k_d := \dim_{\mathbb{C}} W_d.$
We write
\[
SW_d:=\{s\in W_d | \|s\|=1\}
\]
for the unit sphere in $W_d$, equipped with its Haar probability measure $\mu_{W_d}$
(equivalently: $\displaystyle{s=g/\|g\|}$ with $g$ a standard complex Gaussian in $W_d$). Fix an orthonormal basis $\{S_1^d,\dots,S_{k_d}^d\}$ of $W_d$ and define the associated
(subspace) Kodaira map
\[
\Phi_{W_d}:M\to\mathbb{C}^{k_d},
\qquad
\Phi_{W_d}(x):=\big(S_1^d(x),\dots,S_{k_d}^d(x)\big),
\]
with reproducing kernel
\[ \Pi_{W_d}(x,y)
:=\langle \Phi_{W_d}(x), \Phi_{W_d}(y)\rangle 
=\sum_{j=1}^{k_d} S_j^d(x)\overline{S_j^d(y)} .
\]

\subsubsection{Spherical tail bounds in subspaces}

The following lemma is a direct analogue of Lemma~2.1 in
\cite{SZ}, but formulated for the unit sphere in a general
subspace $W_d$, Its proof follows precisely the one of the cited lemma, hence will not be reproduced here. 

\begin{lemma}[Spherical tail bound in $W_d$]\label{lem:subspace-spherical}
Let $A\in W_d$ be a unit vector and let $\mu_{W_d}$ be the Haar-uniform probability measure on $SW_d$.
Then for $0<\lambda<1$,
\[
\mu_{W_d}\big\{ P\in SW_d: |\langle P, A\rangle |>\lambda \big\}
=(1-\lambda^2)^{k_d-1}
\le e^{-(k_d-1)\lambda^2}.
\]
\hfill $\Box$
\end{lemma}
\begin{remark}[Real version and a non-sharp complex bound]
Viewing $W_d\simeq\mathbb{C}^{k_d}$ as a real Euclidean space
$\mathbb{R}^{2k_d}$, L\'evy concentration applied to the
$1$--Lipschitz function $P\mapsto \mathrm{Re}\,\langle P, A\rangle$ yields
\[
\mu\big\{|\mathrm{Re}\,\langle P, A\rangle |>\lambda\big\}
< \exp\!\Big(-\frac{(k_d-1)\lambda^2}{2}\Big).
\]
As in \cite{SZ}, this implies the weaker (non-sharp) estimate
\[
\mu \big\{|\langle P, A\rangle |>\lambda\big\}
\le
2\exp\!\Big(-\frac{(k_d-1)\lambda^2}{4}\Big)\cdot
\]
\end{remark}

\subsubsection{Hyperplane subspaces} \label{hyperplane}
Let us fix a (deterministic) unit vector $v_d\in SH_d$ (e.g.\ a normalized H\"ormander peak section)
and consider the hyperplane
\[
W_d:=v_d^\perp\subset H_d,\qquad k_d:=\dim_{\mathbb{C}}W_d=N_{d}-1\asymp d^m.
\]
The pointwise geometry of $\Pi_{W_d}(x,x)=\Pi_d(x,x)-|v_d(x)|^2$ may degenerate near the peak of $v_d$.
Nevertheless, we claim that the \emph{upper tail} of the sup-norm for a random section in $SW_d$ admits the same
$\sqrt{\log d}$ bound as in \cite{SZ}.  The next theorem is formulated so that it can be inserted
directly into the proof scheme of Shiffman-Zelditch in \cite{SZ}.

\begin{theorem}\label{lem:hyperplane-tail}
Fix any $v_d\in SH_d$ and set $W_d:=v_d^\perp$.
Let $\mu_{W_d}$ be the Haar probability measure on $SW_d$. Then there exists $C>0$ such that
\[
\mu_{W_d}\Big\{s\in SW_d:\ \|s\|_{L^\infty(M)} > C\sqrt{\log d}\Big\}=O(d^{-1}).
\]
and more generally $O(d^{-j})$ for every $j\in\mathbb{Z}_+$ after increasing $C=C(j)$.
\end{theorem}

\begin{proof}
Recall that the ambient Bergman--Szeg\H{o} kernel $\Pi_d$ satisfies the standard uniform estimates (see \cite{Ca}, \cite{Ch}, \cite{SZ2}, \cite{Ze}) :
There exist constants $C_0,C_1>0$ (independent of $N$) such that
\begin{align}
\Pi_d(x,x) &\le C_0\,d^m \qquad \forall x\in X,\ \forall d, \label{eq:Pi_diag_upper}\\
\|K_x-K_y\| &\le C_1\,d^{\frac{m+1}{2}}\,d_M(x,y)
\qquad \text{whenever } d_M(x,y)\le d^{-1/2}, \label{eq:K_Lipschitz}
\end{align}
where $K_x\in H_d$ denotes the evaluation (coherent) vector: $s(x)=\langle s, K_x\rangle$ and $\|K_x\|^2=\Pi_d(x,x)$. Indeed, the second inequality follows from the identity
$$\|K_x-K_y\|^2=\Pi_d(x,x)+\Pi_d(y,y)-2\Re\Pi_d(x,y)$$
and the near-diagonal Bergman kernel expansion
$\Pi_d(x,y)=d^m e^{-cd d_{M}(x,y)^2}(1+O(d^{-1/2}))$,
which implies the estimate
$\Pi_d(x,x)-\Re\Pi_d(x,y)\lesssim d^{m+1}d_{M}(x,y)^2$
for $d_{M}(x,y)\le d^{-1/2}$.

Fix $x\in X$. Let $P_W:H_d\to W_d$ be the orthogonal projection and set
\(
K_x^W:=P_W K_x\in W_d.
\)
Since $s\in W_d$ implies $s(x)=\langle s, K_x\rangle=\langle s, K_x^W\rangle $, we have
\[
|s(x)| = \big|\langle s, K_x^W\rangle \big|.
\]
If $K_x^W\neq 0$ let $\displaystyle{A_x:=K_x^W/\|K_x^W\|\in SW_d}$; otherwise $s(x)=0$ for all $s\in W_d$ and the desired 
estimate is trivial. In the nontrivial case,
\(
|s(x)| = \|K_x^W\|\,|\langle s,A_x\rangle |.
\)

Because $\|K_x^W\|\le \|K_x\|=\sqrt{\Pi_d(x,x)}$, Lemma~\ref{lem:subspace-spherical} (applied on $SW_d$)
gives, for $t>0$,
\begin{equation}\label{eq:point_tail}
\mu_{W_d}\big\{|s(x)|>t\big\}
\le
\exp\!\left(-(k_d-1)\frac{t^2}{\Pi_d(x,x)}\right)
\le
\exp\!\left(-\frac{k_d-1}{C_0 d^m}\,t^2\right),
\end{equation}
where we used \eqref{eq:Pi_diag_upper} in the second inequality.
Since $k_d=N_{d}-1\sim c d^m$, there exist $d_0$ and $c_2>0$ such that for $d\ge d_0$,
\[
\frac{k_d-1}{C_0 d^m}\ge c_2,
\]
hence \eqref{eq:point_tail} yields the uniform subgaussian bound
\begin{equation}\label{eq:point_tail_simple}
\mu_{W_d}\big\{|s(x)|>t\big\}\le e^{-c_2 t^2}\qquad \forall x\in X,\ \forall t>0,\ \forall d\ge d_0.
\end{equation}

Fix $t_d:=A\sqrt{\log d}$ with $A>0$ to be chosen and set
\( \displaystyle{
\delta_d:=\frac{t_d}{4C_1 d^{\frac{m+1}{2}}}
=\frac{A\sqrt{\log d}}{4C_1\,d^{\frac{m+1}{2}}}. }
\)
Choose a $\delta_d$--net $\mathcal{N}_d\subset M$ with cardinality
\( \displaystyle{
|\mathcal{N}_d|\le C_3\,\delta_d^{-2m}
\lesssim d^{m(m+1)}(\log d)^{-m},}
\)
for a constant $C_3$ depending only on $(M,m)$. For any $x,y\in X$ with $d_M(x, y)\le \delta_d\le d^{-1/2}$, we have by Cauchy--Schwarz inequality and \eqref{eq:K_Lipschitz}
\[
|s(x)-s(y)|=\big|\langle s, K_x-K_y\rangle \big|\le \|s\|\,\|K_x-K_y\|\le C_1\,d^{\frac{m+1}{2}}\,d_M(x,y).
\]
Thus for $d_M( x, y)\le\delta_d$,
\begin{equation}\label{eq:osc}
|s(x)-s(y)|\le C_1 d^{\frac{m+1}{2}}\delta_d = \frac14\,t_d.
\end{equation}

\smallskip
\noindent
If $\max_{p\in\mathcal{N}_d}|s(p)|\le t_d$, then for any $x\in M$
choose $p\in\mathcal{N}_d$ with $d(x,p)\le\delta_d$ and obtain
\[
|s(x)|\le |s(p)|+|s(x)-s(p)|
\le t_d+\frac14 t_d=\frac54\,t_d.
\]
Hence,
\begin{equation}\label{eq:reduce}
\mu_{W_d}\left\{\|s\|_\infty>\frac{5}{4} t_d\right\}\le \mu_{W_d}\Big\{\max_{p\in\mathcal{N}_d}|s(p)|>t_d\Big\}.
\end{equation}
Now, by \eqref{eq:point_tail_simple} and a union bound over $\mathcal{N}_d$,
\[
\mu_{W_d}\Big\{\max_{p\in\mathcal{N}_d}|s(p)|>t_d\Big\}
\le |\mathcal{N}_d|\,e^{-c_2 t_d^2}
\le C_4\,d^{m(m+1)}(\log d)^{-m}\,d^{-c_2A^2},
\]
for a constant $C_4$ where we choose $A$ so that $c_2A^2\ge m(m+1)+2$. Then the right-hand side is $O(d^{-1})$. 
Combining with \eqref{eq:reduce} and absorbing fixed constants into $C$ yields
\[
\mu_{W_d}\{\|s\|_\infty>C\sqrt{\log d}\}=O(d^{-1}).
\]

\end{proof}


Let now $M$ be $\mathbb{C}P^2$ with its Fubini--Study K\"ahler form, and let
$L=\mathcal{O}_{\mathbb{C}P^2}(1)$ be the hyperplane bundle with its standard hermitian metric, as described in subsection \ref{FS}.
Then, for each $d\ge 1$, we have
$H_d:=H^0(\mathbb{C}P^2,\mathcal{O}_{\mathbb{CP}^2}(d))$,
equipped with the induced pointwise hermitian product $h_{d}$ and the induced $L^2$-inner product \eqref{hermprod}.

\begin{corollary}
\label{cor:CP2_real_hyperplane}
Fix any unit vector $P\in SH_d$ and define the hyperplane
$W_d:=\langle P\rangle^{\perp} \subset H_d$.  Let $\mu_{W_d}$ be the Haar probability measure on the unit sphere $SW_d$.
Then there exists a constant $C>0$ such that
\[
\mu_{W_d}\Big\{\,Q\in SW_d:\ \|Q\|_{L^\infty(\mathbb{C}P^2)}
> C\sqrt{\log d}\,\Big\}
=O(d^{-1}),
\]
and more generally $O(d^{-j})$ for every $j\in\mathbb{Z}_+$ after increasing $C=C(j)$.
\end{corollary}

The orthogonal direct sum decomposition of complex polynomials in \eqref{realim} as the sum of a polynomial with real
coefficients and one with purely imaginary coefficients induces an orthogonal direct sum decomposition 
$W_d = W_{d,\mathbb{R}} \oplus i\,W_{d,\mathbb{R}}$,
where
$
W_{d,\mathbb{R}}:=\{P\in W_d:\ P \text{ has real coefficients}\}.
$
Analogous to \eqref{realKostlan}, there is an induced real Gaussian probability measure on 
$W_{d,\mathbb{R}}$, denoted again by $\mu_{\mathbb{R}}$, given by 
\begin{equation}
d\mu_{\mathbb{R}}(P)=\frac{e^{-\|P\|^2_{2}}}{\sqrt{\pi}^{\dim_{\mathbb{C}}(W_{d})}}dP 
\end{equation} 

\noindent With the same proof as for Theorem \ref{lem:hyperplane-tail} and Corollary \ref{cor:CP2_real_hyperplane} adapted to the real case, we deduce that there exists a constant $C'>0$ such that
\begin{equation}
\mu_{\mathbb{R}}\Big\{\,P\in SW_{d,\mathbb{R}}\cap (\mathbb{R}P_0)^\perp:
\ \|P\|_{L^\infty(\mathbb{C}P^2)}
> C'\sqrt{\log d}\,\Big\}
=O(d^{-1}),
\end{equation}
and more generally $O(d^{-j})$ after increasing $C'=C'(j)$.

\subsubsection{A formulation for general subspaces} We now formulate the results of subsection \ref{hyperplane} for a subspace of any codimension. The results of this subsection will not be used for geometric applications in this paper, however we include them for completeness.
For a general deterministic subspace $W_d\subset H_d$ we define
 \begin{equation}\label{alphad}
\alpha_d \;:=\; \sup_{x\in M}\frac{\Pi_{W_d}(x,x)}{k_d}
\end{equation}
that controls the pointwise tail in Lemma \ref{lem:subspace-spherical}. Note that by the averaging identity
\[
\int_M \Pi_{W_d}(x,x)\,dV(x)=k_d,
\]
we obtain the lower bound
$\displaystyle{
\alpha_d \;\ge\; \frac{1}{Vol(M)}\cdot
}$

\begin{theorem}\label{thm:alpha}
Let $W_d\subset H_d$ be a complex subspace of dimension $k_d\asymp d^m$, and let
$\mu_{W_d}$ be the Haar probability measure on $SW_d$.
Define $\alpha_{d}$ by \eqref{alphad}.
Then there exists $C>0$ such that
\[
\mu_{W_d}\Big\{s\in SW_d:\ \|s\|_{L^\infty(M)} > C\,\sqrt{\alpha_d\,\log d}\Big\}
=O(d^{-1}),
\]
and more generally $O(d^{-j})$ for every $j\in\mathbb{Z}_+$ after increasing $C=C(j)$.
\end{theorem}

\begin{proof}[Proof sketch]
Fix $x\in M$ and set $K_x^W:=P_{W_d}K_x\in W_d$, so that for $s\in W_d$,
\(
s(x)=\langle s,K_x\rangle=\langle s,K_x^W\rangle.
\)
If $K_x^W\neq 0$ write $K_x^W=\|K_x^W\|A_x$ with $A_x\in S(W_d)$.
Then
\(
|s(x)|=\|K_x^W\|\,|\langle s,A_x\rangle|.
\)
Applying Lemma~\ref{lem:subspace-spherical} on the sphere $SW_d$ gives
\[
\mu_{W_d}\{|s(x)|>t\}\le 
\exp\!\Big(-(k_d-1)\frac{t^2}{\|K_x^W\|^2}\Big).
\]
Since $\|K_x^W\|^2=\Pi_{W_d}(x,x)\le k_d\alpha_d$, we obtain the uniform bound
\[
\mu_{W_d}\{|s(x)|>t\}
\le
\exp\!\Big(-c\,\frac{t^2}{\alpha_d}\Big),
\]
with $c>0$ independent of $x$ and $d$.

\smallskip
To pass from a fixed point to the supremum, fix the threshold
\(\displaystyle{
t:= C\,\sqrt{\alpha_d\,\log d},}
\)
and choose a $\delta$--net $\mathcal{N}_\delta\subset M$ with
$|\mathcal{N}_\delta|\lesssim \delta^{-2m}$.
Using the ambient coherent vectors, we have
\(\displaystyle{
|s(x)-s(y)|
=\big|\langle s,K_x-K_y\rangle\big|
\le \|K_x-K_y\|
}\). The standard near-diagonal estimate
$
\|K_x-K_y\|\lesssim d^{\frac{m+1}{2}}\,d_{M}(x,y)$
holds for $d_{M}(x,y)\le d^{-1/2}$. 
We choose
\[
\delta \asymp \frac{t}{d^{\frac{m+1}{2}}}
\asymp \frac{\sqrt{\alpha_d\log d}}{d^{\frac{m+1}{2}}}\cdot
\]
With this choice, whenever $d(x,p)\le\delta$ we have
\(\displaystyle{
|s(x)-s(p)|
\le \|K_x-K_p\|
\lesssim d^{\frac{m+1}{2}}\delta
\le \frac14\,t,
}\)
so that if $\max_{p\in\mathcal{N}_\delta}|s(p)|\le t$ then
\(\displaystyle{
\|s\|_{L^\infty(M)}\le \frac54\,t.}
\)
Moreover, the net cardinality satisfies
\[
|\mathcal{N}_\delta|
\lesssim \delta^{-2m}
\lesssim
\Big(\frac{d^{\frac{m+1}{2}}}{\sqrt{\alpha_d\log d}}\Big)^{2m}
=
\frac{d^{m(m+1)}}{(\alpha_d\log d)^m}.
\]
Therefore, using the pointwise tail estimate and a union bound, we obtain the inequality
\[
\mu_{W_d}\Big\{\max_{p\in\mathcal{N}_\delta}|s(p)|>t\Big\}
\le
|\mathcal{N}_\delta|\;e^{-c\,t^2/\alpha_d}
=
|\mathcal{N}_\delta|\;e^{-cC^2\log d}.
\]
Since $|\mathcal{N}_\delta|$ grows polynomially in $d$, choosing $C$ large enough yields
\[
\mu_{W_d}\Big\{\|s\|_{L^\infty(M)}>C\,\sqrt{\alpha_d\log d}\Big\}
=O(d^{-1}),
\]
and similarly $O(d^{-j})$ for arbitrary $j$ after increasing $C=C(j)$.
\end{proof}

In particular, when $W_d=H_d$ (resp. $W_{d}$ is a hyperplane in $H_{d}$) we have $\alpha_d \approx1$ and we recover the original statement of \cite{SZ} (resp. the statement of Theorem \ref{lem:hyperplane-tail}).

\section{A variant of the barrier method}
\subsection{An observation about components of real curves in an annulus} \label{annulus}
Suppose now that $P$ is a real homogeneous polynomial of degree $d$ in $X_{0}, X_{1}, X_{2}$, hence $\{P=0\}\cap \mathbb{R}P^2$ is a real algebraic curve. Consider the restriction of this polynomial and curve to $\mathbb{R}P^2\cap U_{0}\cong \mathbb{R}^2$. Let $r_{1}, r_{2}$ be real numbers such that $0<r_{1}<r_{2}$ and let $A_{r_{1}, r_{2}}=\{z=(z_{1},z_{2})\in \mathbb{R}P^2\cap U_{0}| r_{1}<\|z\|<r_{2}\}$ denote the annulus bounded by circles of (Euclidean) radii $r_{1}$ and $r_{2}$ centered at $(0,0)$. We assume that $P(z)\neq 0$ if $\|z\|=r_{1}$ or if $\|z\|=r_{2}$, namely the polynomial $P$ does not vanish on the boundary of the annulus.  Since $A_{r_{1}, r_{2}}$ deformation retracts to a circle, we observe that $H_{1}(A_{r_{1}, r_{2}}, \mathbb{Z}_{2})\cong \mathbb{Z}_{2}$. 
The zero locus $\{P=0\}\cap A_{r_{1},r_{2}}$ represents either the trivial or the non-trivial class in this homology group. The following result provides a sufficient condition for such a class to remain invariant under perturbation. 

\begin{proposition} \label{homotopy} 
Suppose that $Q$ is a real polynomial of degree $d$ such that 
\[ \|Q(z)\|_{FS} < \|P(z)\|_{FS} \quad \mathrm{for}\, \mathrm{all} \, z\,  \mathrm{with} \, \|z\|=r_{1}\, \mathrm{or} \, \|z\|=r_{2}. \]  
Then the homology classes of $\{P=0\}\cap A_{r_{1},r_{2}}$ and $\{P+Q=0\}\cap A_{r_{1},r_{2}}$ in $H_{1}(A_{r_{1}, r_{2}}, \mathbb{Z}_{2})\cong \mathbb{Z}_{2}$ are the same.
\end{proposition} 

We remark that the analogous result for integer homology does not hold if all homologically non-trivial components of $P=0$ are understood to represent the same generator of $H_{1}(A_{r_{1}, r_{2}}, \mathbb{Z})\cong \mathbb{Z}$. . For instance, if the zero locus of $P$ is two circles centered at the origin contained in $A_{r_{1},r_{2}}$ with very close radii, then even under a very small perturbation the two circles may merge and then disappear. 

\noindent \textit{Proof:}  Consider the homotopy $P_{t}=P+tQ$ between $P$ and $P+Q$, where $t\in [0,1]$. We observe that if $\|z\|=r_{1}$ or $\|z_{2}\|=r_{2}$ then 
\[ \|(P+tQ)(z)\|_{FS} \geq \|P(z)\|_{FS}-t \|Q(z)\|_{FS}\geq  \| P(z)\|_{FS}-\|Q(z)\|_{FS} > 0. \]Hence, $P+tQ$ cannot vanish on the boundary of $A_{r_{1},r_{2}}$ for any value of $t\in [0,1]$. For any $(\alpha,\beta)$ on the unit circle in $\mathbb{R}^2$, consider the intersection of the ray $R_{\alpha,\beta}=\{(\alpha s, \beta s) | s\in \mathbb{R}_{+}\}$ with $A_{r_{1},r_{2}}$. Since $P+tQ$ cannot vanish on its boundary, we conclude that the zeros of $P+tQ$ on $R_{\alpha,\beta}\cap A_{r_{1},r_{2}}$ as $t$ increases are either created or annihilated in pairs (counted with multiplicity). Hence the parities of the number of zeros of $P$ and $P+Q$ on $R_{\alpha,\beta}\cap A_{r_{1},r_{2}}$ are the same. The result then follows by noting that the class of the zero locus of a polynomial $P+tQ$  in $H_{1}(A_{r_{1}, r_{2}}, \mathbb{Z}_{2})\cong \mathbb{Z}_{2}$ agrees with the parity of the number of zeros (counted with multiplicity) of $P+tQ$ on $R_{\alpha,\beta}\cap A_{r_{1},r_{2}}$ for every $(\alpha,\beta)$. $\hfill$  $\Box$

\subsection{The barrier argument and lower estimation of probabilities} \label{barrier}
Let the annulus $A_{r_{1},r_{2}}$ be as in subsection \ref{annulus}, whereas $g(d)$, $R(d)$ and $K$ are as in subsection \ref{meanvalue}. Let us assume that $\{\|z\|=r_{1}\}\cup \{\|z\|=r_{2}\}\subset K$ and that $P$ is a reference polynomial of degree $d$ not vanishing on any point of $K$. Let $W=\langle P\rangle^{\perp}$ in the hermitian space $H_{d}=H^{0}(\mathbb{C}P^2,\mathcal{O}_{\mathbb{C}P^2}(d))$ and $\pi_{W}$, $\pi_{W^{\perp}}=\pi_{\langle P\rangle}$ the orthogonal projections from $H_{d}$ to $W$ and $W^{\perp}$ respectively. Let $Y_{d}$ denote the ball-average norm random variable of Definition \ref{ballaverage} and $Y_{d,W}:=Y_{d}\circ \pi_{W}$, $Y_{d,W^{\perp}}:=Y_{d}\circ \pi_{W^{\perp}}$. We borrow the following lemma from \cite{KW}, omitting its proof. 

\begin{lemma} \label{directsum} 
(Lemma 4.13 of \cite{KW}) 
For every $d>0$, $\mathbb{E}(Y_{d}^2)=N_{d}$ and $\mathbb{E}(Y_{d}^2)=\mathbb{E}(Y_{d,W}^2)+\mathbb{E}(Y_{d,W^{\perp}}^2)$. \hfill $\Box$
\end{lemma} 

\begin{corollary} \label{onehalf}
Consider the induced Gaussian probability measure on $W$ by the Kostlan measure. Then,
\[  \mathbb{\mu}(Q\in W | \|Q\|^2_{B_{FS}(0,\rho(d))}\leq 2N_{d})\geq \frac{1}{2}. \]
\end{corollary} 

\noindent \textit{Proof:} Note that the expected value of $\|Q\|^2_{B_{FS}(0,\rho_{d})}$ over $W$ is equal to $\mathbb{E}(Y_{d,W}^2)$, which is bounded above by $N_{d}$ due to Lemma \ref{directsum}. The result then follows from Markov's inequality. \hfill $\Box$ 

\begin{proposition} \label{aboundprop}
Suppose that $(a(d),Q)\in \mathbb{R_{+}}\times W$ satisfies 
\begin{equation} \label{abound}
a(d)^2\geq \frac{2^4  \exp\left(3g(d)/4\right) \left(1+g(d)/d\right)^3 }{\inf_{z\in K}  \|P(z)\|^2_{FS} }  \|Q\|^2_{B_{FS}(0,\rho_{d})}.  
\end{equation}

\noindent Then the homology classes of $\{P=0\}\cap A_{r_{1},r_{2}}$ and $\{a(d)P+Q=0\}\cap A_{r_{1},r_{2}}$ in $H_{1}(A_{r_{1}, r_{2}}, \mathbb{Z}_{2})\cong \mathbb{Z}_{2}$ are the same.
\end{proposition} 

\noindent\textit{Proof:} By Proposition \ref{homotopy}, it is enough to show that $\displaystyle{ \frac{\|Q(z)\|_{FS}}{a(d)}\leq \|P(z)\|_{FS}}$ on the boundary of $A_{r_{1},r_{2}}$, which follows from Proposition \ref{localsupbound}, noting that the boundary of $A_{r_{1},r_{2}}$ is contained in $K$. \hfill $\Box$  

\begin{corollary} \label{lowerprob}
Let $F$ be a Kostlan random real polynomial of degree $d$. With the assumptions and notation in this subsection, suppose that the reference polynomial $P$ is $L^2$-normalized, namely, $\|P\|_{2}=1$.  Let  $m(d)\in \mathbb{R}_{+}$ be defined by
\begin{equation} \label{md}
m(d)^2:= \frac{2^4  \exp\left(3g(d)/4\right) \left(1+g(d)/d\right)^3 }{\inf_{z\in K}  \|P(z)\|^2_{FS} }2N_{d}.
\end{equation}
 Then the probability that $\{P=0\}\cap A_{r_{1},r_{2}}$ and $\{F=0\}\cap A_{r_{1},r_{2}}$ represent the same class in $H_{1}(A_{r_{1}, r_{2}}, \mathbb{Z}_{2})\cong \mathbb{Z}_{2}$ is bounded below by the quantity
\[ \frac{m(d)}{2\sqrt{2\pi}} e^{-2m(d)^2}. \] 
\end{corollary} 

\noindent \textit{Proof:} By Proposition \ref{aboundprop}, if  $(a(d),Q)\in \mathbb{R_{+}}\times W$ satisfies \eqref{abound}, then $F=a(d)P+Q\in H_{d}$ belongs to the desired set. We next estimate the probability of the set of such polynomials from below. Consider the subset of this set containing  $(a(d),Q)\in \mathbb{R_{+}}\times W$ such that $\|Q\|_{B_{FS}(0,\rho_{d})}\leq 2N_{d}$ and $\displaystyle{a(d)\geq m(d)}$. On the corresponding rectangular subset $S$ containing $a(d)P+Q$ of $H_{d}$, the Gaussian probability density decomposes as a product. By Corollary \ref{onehalf}, we know that the probability that $Q$ satisfies the given condition is at least $1/2$. On the other hand, along $Span(P)$, the Gaussian will be a standard one since $\|P\|_{2}=1$ and the corresponding probability is 
\begin{eqnarray*} 
\frac{1}{\sqrt{2\pi}}\int_{m(d)}^{\infty} e^{-x^2/2}dx &\geq&  \frac{1}{\sqrt{2\pi}}\int_{m(d)}^{2m(d)} e^{-x^2/2}dx  \\ 
&\geq& \frac{m(d)}{\sqrt{2\pi}} e^{-(2m(d))^2/2} 
\end{eqnarray*}
The result then follows. \hfill $\Box$

\section{Applications}
\subsection{Connected components of length larger than $O(\sqrt{1/d})$} \label{reference} 
Let $f:\mathbb{N^{+}}\rightarrow \mathbb{R}_{\geq 0}$ be a monotone increasing function such that $1 \ll f(d)\ll d$ for $d$ large. 
Let us define the reference polynomial 
\[ P_{0}(X_{0},X_{1},X_{2})= X_{0}^{d-2}\left(X_{1}^2+X_{2}^2-\left( \frac{f(d)}{d} \right)      X_{0}^2\right). \] 
The zero locus of $P_{0}$ in $\mathbb{R}P^2$ is the union of a $(d-2)$-fold line at infinity together the zero locus of $P_{0}$ on $\mathbb{R}P^2\cap U_{0}$. The latter is a circle of radius $\displaystyle{\sqrt{f(d)/d}}$ centered at the origin. 

\begin{lemma} \label{L2estimate}
The $L^2$-Fubini-Study norm of $P_{0}$ is given by 
\[ \|P_{0}\|_{2}^2=      \frac{2 f(d)^2}{d^4} +O\left(\frac{1}{d^4}\right). \]                           
\end{lemma} 

\noindent\textit{Proof:} By Lemma \ref{L2norms}, the monomials $X_{0}^{d}$, $X_{0}^{d-2}X_{1}^2$ and $X_{0}^{d-2}X_{2}^2$ are mutually orthogonal. Furthermore, by the same lemma, 
\[ \|X_{0}^{d}\|_{2}^2=\frac{2}{(d+2)(d+1)},              \qquad  \|X_{0}^{d-2}X_{1}^2\|_{2}^2=\|X_{0}^{d-2}X_{2}^2\|_{2}^2= \frac{4}{(d+2)(d+1)d(d-1)}\cdot                                 \]       
The result then follows from the Pythagorean theorem. $\hfill$ $\Box$ 

Let us now restrict $P_{0}$ to an annulus $A_{r_{1},r_{2}}$ where $\displaystyle{r_{1}=\sqrt{\frac{f(d)}{2d}}}$ and $\displaystyle{r_{2}=\sqrt{\frac{3f(d)}{2d}}}$, whose zero locus in this annulus is the circle centered at the origin with radius $\displaystyle{\sqrt{\frac{f(d)}{d}}}$. We would like to investigate perturbations of $P_{0}$ for which we can use the homology invariance argument in Proposition \ref{homotopy} and in particular Corollary \ref{lowerprob}. 
For this purpose, we next estimate the pointwise Fubini-Study norm of $P_{0}$ from below, on the boundary of $A_{r_{1},r_{2}}$. 

\begin{lemma} \label{pointwiseestimate}
For $\displaystyle{\|z\|=r_{1}=\sqrt{\frac{f(d)}{2d}}}$ and $d$ sufficiently large,  
\( \displaystyle{\|P_{0}(z)\|_{FS}^2 \geq  \frac{f(d)^2}{8d^2} \exp\left(  -\frac{f(d)}{2}\right)}\)
whereas for $\displaystyle{ \|z\|=r_{2}= \sqrt{\frac{3f(d)}{2d}}}$ and $d$ sufficiently large, 
\(\displaystyle{ \|P_{0}(z)\|_{FS}^2 \geq  \frac{f(d)^2}{8d^2} \exp\left(  -\frac{3f(d)}{2}\right)}\cdot\)
\end{lemma}    

\noindent\textit{Proof:} Assume that $\|z\|=r_{1}$. By  $\eqref{pointwise}$, 
\begin{eqnarray} 
\|P_{0}(z)\|_{FS}^2 &=& \frac{|P_{0}(1,z_{1},z_{2})|^2}{(1+\|z\|^2)^d} \nonumber \\
&=& \frac{ \left| \|z\|^2-\frac{f(d)}{d}\right|^2}{(1+\|z\|^2)^d} \nonumber\\
&=&      \frac{ \left|  \frac{f(d)}{2d}-\frac{f(d)}{d}  \right|^2     }{ \left(1+   \frac{f(d)}{2d}\right)^d } \nonumber    \\
&=&  \frac{ f(d)^2}{4d^2} \exp\left(-d\log\left(1+\frac{f(d)}{2d}\right)\right) \label{infbound} \\
&\geq&  \frac{f(d)^2}{8d^2} \exp\left(  -\frac{f(d)}{2} \right) \nonumber
\end{eqnarray}
where the last inequality is deduced for $d$ sufficiently large from the assumption $f(d)\ll d$ and the fact that $\log(1+x)\sim x$ for $x$ small. The proof of the second statement is similar. \hfill $\Box$

Let $\displaystyle{\widetilde{P_{0}}}=P_{0}/\|P_{0}\|_{2}$ denote the $L^2$-normalization of $P_{0}$, so that $\displaystyle{\widetilde{P_{0}}\in SH_{d}}$. The following result is an immediate corollary of Lemmas \ref{L2estimate} and \ref{pointwiseestimate}.

\begin{corollary} \label{normalizedpointwise}
For $\displaystyle{\|z\|=r_{1}=\sqrt{\frac{f(d)}{2d}}}$ and $d$ sufficiently large,  
\( \displaystyle{\|\widetilde{P_{0}}(z)\|_{FS}^2 \geq \frac{d^{2}}{32}\exp\left(  -\frac{f(d)}{2}\right)}\)
whereas for $\displaystyle{ \|z\|=r_{2}= \sqrt{\frac{3f(d)}{2d}}}$  and $d$ sufficiently large, 
\(\displaystyle{ \|\widetilde{P_{0}}(z)\|_{FS}^2 \geq \frac{d^{2}}{32}\exp\left(  -\frac{3f(d)}{2}\right)\cdot}\) 
\hfill $\Box$
\end{corollary}
\begin{remark}
Note that $\widetilde{P_{0}}\in SH_{d}$ has an unusually large $L_{\infty}$-norm $\|\widetilde{P_{0}}\|_{\infty}$ in view of Theorem \ref{thm:alpha}, which for the case of $W_{d}=H_{d}$ shows that the set of polynomials in $SH_{d}$ with their $L_{\infty}$-norms exceeding $\sqrt{\log(d)}$ in order has very small probability.  On the contrary, it can be shown that, in the Hilbert space $H_{d}$, the angle between $\widetilde{P_{0}}$ and a Hörmander peak section centered at the point $[1:0:0]\in \mathbb{RP}^2$  is very small, hence providing a conceptual explanation of why $\|\widetilde{P_{0}}\|_{\infty}$ is large.  
\end{remark}

\begin{theorem} \label{largecomponents}
Let $\epsilon>0$. Then, there exists a positive constant $c$ independent of $d$ such that for $d$ sufficiently large,  the expected number of connected components of a Kostlan random real plane curve which have length at least $\displaystyle{2\pi\sqrt{\frac{\log\left(\epsilon \log(d)/2^{11}\right)}{12d}}}$ is bounded below by $\displaystyle{\frac{cd^{1-\epsilon}}{\log(\log(d))}\cdot}$
\end{theorem} 

\noindent\textit{Proof:} Let us set $\displaystyle{f(d)=\frac{1}{6}\log(\epsilon \log(d)/2^{11})}$, and using the notation of subsection \ref{meanvalue}, let $g(d)=6f(d)$ (hence, $\displaystyle{R(d)=\sqrt{6f(d)/d}}$) and let $K=\partial(A_{r_{1},r_{2}})$. Set the reference polynomial to be $\widetilde{P}_{0}(z)$ above, which is $L^2$-normalized. We now estimate $m(d)$ defined by \eqref{md}. 

\begin{eqnarray*} 
m(d)^2&=&  \frac{2^4  \exp\left(3g(d)/4\right) \left(1+g(d)/d\right)^3 }{\inf_{z\in K}  \|\widetilde{P}_{0}(z)\|^2_{FS} }2N_{d}    \\
&\leq&  \frac{2^4  \exp\left(9f(d)/2\right) \left(1+6f(d)/d\right)^3 }{(d^{2}/32)\exp\left(  -3f(d)/2\right) }(d+2)(d+1) \, \mathrm{for} \, d \, \mathrm{large},\, \mathrm{by} \, \mathrm{Corollary} \, \ref{normalizedpointwise}       \\
&\leq& 2^{10}\exp(6f(d)) \, \mathrm{for} \, d \, \mathrm{large},\, \mathrm{since} \, f(d)\ll d. 
\end{eqnarray*}
\noindent On the other hand, by using Lemma \ref{L2estimate} and equation \eqref{infbound} it is easy to see that $m(d)$ grows to infinity with $d$. 

Let $F$ be a Kostlan random real polynomial of degree $d$. Let $p$ denote the probability that $\{\widetilde{P}_{0}=0\}\cap A_{r_{1},r_{2}}$ and $\{F=0\}\cap A_{r_{1},r_{2}}$ represent the same class in $H_{1}(A_{r_{1}, r_{2}}, \mathbb{Z}_{2})\cong \mathbb{Z}_{2}$. The first of these classes is the non-trivial one, hence the probability that the other one is also non-trivial is at least 
\begin{eqnarray*} 
p&\geq& \frac{m(d)}{2\sqrt{2\pi}} e^{-2m(d)^2}\\
&\geq&e^{-2^{11}\exp(6f(d))} \\
&\geq& d^{-\epsilon}. 
\end{eqnarray*} 

\noindent for $d$ sufficiently large, by Corollary \ref{lowerprob}  and the bounds obtained for $m(d)$ above. Since the shortest curve in $A_{r_{1},r_{2}}$ realizing the non-trivial homology class in $H_{1}(A_{r_{1}, r_{2}}, \mathbb{Z}_{2})$ is the inner circle boundary, which has radius $\displaystyle{r_{1}=\sqrt{\frac{f(d)}{2d}}}$, we deduce that with probability at least $d^{-\epsilon}$, there exists a component of the curve in $A_{r_{1},r_{2}}$ with length at least $\displaystyle{2\pi r_{1}= 2\pi\sqrt{\frac{\log\left(\epsilon \log(d)/2^{11}\right)}{12d}}\cdot}$ 

In order to establish the statement about the expected number of such components, let us pack disjoint balls of Fubini-Study radius $\arctan(R(d))$ inside $\mathbb{R}P^2$. Since for $x$ small $\tan(x)\sim x$, we can use $R(d)$ itself for the asymptotics instead. We can pack altogether $\displaystyle{O\left( \frac{1}{R(d)^2}\right)=O\left(\frac{d}{\log(\log(d))}\right)}$ such balls. The expected number of the aforementioned components in one ball is at least $d^{-\epsilon}$ by above and the final result follows by the linearity of expectation. \hfill $\Box$

\subsection{Nests with depth at least $O(\log(\log(d)))$}  \label{Chebyshevimprove}
The goal of this subsection is to use the properties of Chebyshev polynomials recalled in subsection \ref{subsectionChebyshev} in order to obtain a new reference polynomial, which will then allow us to prove a lower bound on the expected number of nests of a  Kostlan random real algebraic plane curves having depth $O(\log(\log(d)))$, which tends to infinity as $d$ grows. 

Let $f:\mathbb{N^{+}}\rightarrow \mathbb{R}_{\geq 0}$ be a monotone increasing function such that $1 \ll f(d)\ll d$ for $d$ large, as in subsection \ref{reference}, and let $\alpha\in (0,1)$ be a constant. Let $N=\alpha \lfloor f(d) \rfloor$ or $N=\alpha \lfloor f(d) \rfloor-1$ (where $\lfloor \cdot \rfloor$ denotes the floor function), where the choice is made so that $N$ becomes an even integer. Set 
\[ p_{N}(x)=T_{N}\left(\frac{d}{f(d)} x\right) = \sum_{j=0}^{N} a_{j,N} \left( \frac{d}{f(d)}\right)^{j} x^{j}. \]
By Proposition \ref{roots}, the polynomial $p_{N}(x)$ has all its roots in the interval $\displaystyle{ \left[ -\frac{f(d)}{d}, \frac{f(d)}{d} \right]}$.  In fact, these roots are located at $\displaystyle{ x=\frac{f(d)}{d} \cos\left(\frac{(2k+1)\pi}{2N}\right)}$ for $ k\in \{0,1,\ldots,N-1\}$. The extremal values of $p_{N}(x)$ in this interval are $\pm 1$ and these are located at $\displaystyle{ x=\frac{f(d)}{d} \cos\left(\frac{k\pi}{N}\right)}$ for  $ k\in \{0,1,\ldots, N\}$. 

Let us define the reference polynomial $P_{1}(X_{0},X_{1},X_{2})$ by substituting $\displaystyle{\frac{X_{1}^2}{X_{0}^2}+\frac{X_{2}^2}{X_{0}^2}}$ for $x$ in $p_{N}(x)$ and then homogenizing to degree $d$. More precisely, 
\begin{eqnarray} 
P_{1}(X_{0},X_{1}, X_{2})&=& X_{0}^{d-2N} \sum_{j=0}^{N} a_{j,N} \left( \frac{d}{f(d)} \right)^j (X_{1}^2+X_{2}^2)^{j} X_{0}^{2N-2j}    \\
&=& \sum_{j=0}^{N} \sum_{\ell=0}^{j} a_{j,N} \left( \frac{d}{f(d)} \right)^j   {j \choose \ell}  X_{1}^{2\ell} X_{2}^{2j-2\ell} X_{0}^{d-2j}  \label{P1}
\end{eqnarray} 

\noindent The zero locus of $P_{1}$ is then the union of a $(d-2N)$-fold line at infinity and its zero locus in $\mathbb{R}P^2\cap U_{0}$. By construction, the latter is a union of $N/2$ circles, all centered at $[1:0:0]$ and having Euclidean radii $\displaystyle{R_{k}=\sqrt{ \frac{f(d)}{d} \cos\left(\frac{(2k+1)\pi}{2N}\right)}}$ for $ \displaystyle{k\in \left\{0,1,\ldots,\frac{N}{2}-1\right\} } $. The extremal values of $P_{1}$ which lie within a radius of $\displaystyle{\sqrt{\frac{f(d)}{d}}}$ of the center $[1:0:0]$ alternate between $+1$ and $-1$ and they occur on circles of radii  $\displaystyle{r_{k}= \sqrt{\frac{f(d)}{d} \cos\left(\frac{k\pi}{N}\right)}}$ for  $\displaystyle{ k\in \left\{0,1,\ldots, \frac{N}{2}\right\}}$ (the last of these is a degenerate circle of radius $0$). 

\begin{lemma} \label{P1L2}
For $d$ sufficiently large, the $L^2$-norm of $P_{1}$ satisfies the inequality 
\[  \|P_{1}\|_{2}^2 \leq   \frac{f(d)^2\cdot  2^{10f(d)}}{d^2}  \cdot   \] 
\end{lemma} 

\noindent \textit{Proof:} By using equation \eqref{P1} and Lemma \ref{L2norms} we get 
\[  \|P_{1}\|_{2}^2 = \sum_{j=0}^{N} \sum_{\ell=0}^{j}   |a_{j,N}|^2  \frac{d^{2j}}{(f(d))^{2j}} \frac{ (j!)^2}{(\ell !)^2((j-\ell)!)^2} \frac{(d-2j)! (2\ell)! (2j-2\ell)! 2}{(d+2)!}\cdot  \] 
We use the upper estimate $\displaystyle{|a_{j,N}|\leq 2^{j}{N\choose j}}$ given by Lemma \ref{coefficient}. Since $\alpha<1$, we get $N\leq \alpha\lfloor f(d)\rfloor < f(d)$. Hence $\displaystyle{  \frac{ (j!)^2}{(f(d))^{2j}}\leq 1}$. Also, for $d$ sufficiently large, $\displaystyle{ \frac{d^{2j}(d-2j)!}{(d+2)!}\leq \frac{2^N}{(d+2)(d+1)}}$ since $j\leq N<f(d)\ll d$.  We deduce that 
\[  \|P_{1}\|_{2}^{2} \leq \sum_{j=0}^{N} \sum_{\ell=0}^{j}   \frac{2^{2j+N+1}}{(d+2)(d+1)} {N \choose j}^2 {2\ell \choose \ell} {2j-2\ell \choose j-\ell}\cdot  \]
By using the inequalities $\displaystyle{2^{2j}\leq 2^{2N}}$,  $\displaystyle{ {2\ell \choose \ell}\leq 2^{2N}}$, $\displaystyle{ {2j-2\ell \choose j-\ell}\leq 2^{2N}}$ and $\displaystyle{ {N \choose j}\leq 2^N}$ we obtain 
\begin{eqnarray*} 
\|P_{1}\|_{2}^2 &\leq&    \sum_{j=0}^{N} \sum_{\ell=0}^{j}   \frac{2^{9N+1}}{(d+2)(d+1)}   \\
&\leq&   \frac{N^2\cdot  2^{9N+1}}{(d+2)(d+1)}  \\
&\leq&  \frac{f(d)^2\cdot  2^{10f(d)}}{d^2} \cdot \\
\end{eqnarray*}
\noindent The result follows. \hfill $\Box$ 

Let $\displaystyle{\widetilde{P}_{1}=P_{1}/\|P_{1}\|_{2}}$ denote the $L^2$-normalization of $P_{1}$, so that $\widetilde{P}_{1}\in SH_{d}$.  

\begin{corollary} \label{P1FS}
Let $\displaystyle{r_{k}= \sqrt{\frac{f(d)}{d} \cos\left(\frac{k\pi}{N}\right)}}$ for  $\displaystyle{ k\in \left\{0,1,\ldots, \frac{N}{2}\right\}}$. Then for every such $k$, for $z\in \mathbb{R}P^2\cap U_{0}$ such that $\|z\|=r_{k}$, and for $d$ sufficiently large, 
\[  \|\widetilde{P}_{1}(z)\|_{FS}^2\geq \frac{d^2}{(f(d))^2}\exp(-9f(d))\cdot \] 
\end{corollary} 

\noindent\textit{Proof:} 
\begin{eqnarray*} 
\|\widetilde{P}_{1}(z)\|_{FS}^2 &=& \frac{\|P_{1}(z)\|^2_{FS}}{\|P_{1}\|^2_{2}}.   \\
&=& \frac{ |P_{1}(1,z_{1},z_{2})|^2}{\|P_{1}\|_{2}^2 (1+\|z\|^2)^d} \quad \mathrm{by} \, \eqref{pointwise} \\
&\geq& \frac{d^2}{f(d)^2 2^{10f(d)} \left(1+ \frac{f(d)}{d} \cos\left(\frac{k\pi}{N}\right) \right)^d} \, \mathrm{by}\, \mathrm{Lemma} \, \ref{P1L2} \\ 
&\geq& \frac{d^2}{f(d)^2 2^{10f(d)} \left(1+ \frac{f(d)}{d}\right)^d}  \\
&=& \frac{d^2}{f(d)^2 2^{10f(d)}} \exp\left(-d\log\left(1+\frac{f(d)}{d}\right)\right) \\
&\geq& \frac{d^2}{f(d)^2 2^{11f(d)}}\exp(-f(d)) \quad \mathrm{since} \, f(d)\ll d \,\mathrm{and}\, \log\left(1+\frac{f(d)}{d}\right) \sim f(d) \\
&\geq& \frac{d^2}{f(d)^2}\exp(-9f(d))\cdot
\end{eqnarray*}
\hfill $\Box$ 

\begin{theorem} \label{expecteddepth}
Let $\epsilon>0$. Then, there exists a positive constant $c$ independent of $d$ such that for $d$ sufficiently large,  the expected number of nests of a Kostlan random real plane curve which have depth at least $\displaystyle{\frac{1}{30}\log(\epsilon \log(d)/2)}$ is bounded below by $\displaystyle{\frac{cd^{1-\epsilon}}{\log(\log(d))}\cdot}$
\end{theorem} 

\noindent \textit{Proof:}  Let us set $\displaystyle{f(d)=\frac{1}{13}\log(\epsilon \log(d)/2)}$, and with the notation of subsection \ref{meanvalue}, let $g(d)= 4f(d)$ (hence, $\displaystyle{R(d)=2r_{0}=\sqrt{\frac{4f(d)}{d}}}$). Let $\displaystyle{K= \bigcup_{k=0}^{N/2} \{\|z\|=r_{k} \}}$, a union of $\displaystyle{\frac{N}{2}+1}$ circles all centered at $[1:0:0]$ where $\displaystyle{r_{k}= \sqrt{\frac{f(d)}{d} \cos\left(\frac{k\pi}{N}\right)}}$ for  $\displaystyle{ k\in \left\{0,1,\ldots, \frac{N}{2}\right\}}$ as above. Note that $r_{0}>r_{1}>\ldots>r_{N/2}=0$. Set the reference polynomial to be $\widetilde{P}_{1}(z)$ above, which is $L^2$-normalized. We next estimate $m(d)$ defined by \eqref{md}. 

\begin{eqnarray*} 
m(d)^2&=&  \frac{2^4  \exp\left(3g(d)/4\right) \left(1+g(d)/d\right)^3 }{\inf_{z\in K}  \|\widetilde{P}_{1}(z)\|^2_{FS} }2N_{d}    \\
&\leq&  \frac{2^4  \exp\left(3f(d)\right) \left(1+4f(d)/d\right)^3 }{ (d^2/f(d)^2)\exp(-9f(d)) }(d+2)(d+1) \, \mathrm{for} \, d \, \mathrm{large},\, \mathrm{by} \, \mathrm{Corollary} \, \ref{P1FS}       \\
&\leq& \exp(13f(d)) \quad \mathrm{for} \, d \, \mathrm{large},\, \mathrm{since} \, 1\ll f(d)\ll d. 
\end{eqnarray*}
\noindent On the other hand, by using Lemma \ref{L2estimate} and equation \eqref{infbound} one sees that $m(d)$ grows to infinity with $d$. 

Let $F$ be a Kostlan random real polynomial of degree $d$. Let $q$ denote the probability that $\{\widetilde{P}_{1}=0\}\cap A_{r_{k+1},r_{k}}$ and $\{F=0\}\cap A_{r_{k+1},r_{k}}$ represent the same class in $H_{1}(A_{r_{k+1}, r_{k}}, \mathbb{Z}_{2})\cong \mathbb{Z}_{2}$ for every $\displaystyle{k\in\left\{0,1,\ldots,\frac{N}{2}-1\right\}}$. For every value of $k$, the first of these classes is the non-trivial one. Corollary \ref{lowerprob} can be applied to all $A_{r_{k+1},r_{k}}$ simultaneously since all of the boundary circles $\{\|z\|=r_{k}\}$ belong to $K$ and a common reference polynomial $\widetilde{P}_{1}$ is used for all of them. Hence, for sufficiently large values of $d$, the probability $q$ can be lower estimated as below. 
\begin{eqnarray*} 
q&\geq& \frac{m(d)}{2\sqrt{2\pi}} e^{-2m(d)^2}\\
&\geq&e^{-2\exp(13f(d))} \\
&\geq& d^{-\epsilon}. 
\end{eqnarray*} 
The homology invariance condition above ensures that the curve $\{F=0\}$ has a nest of depth at least  $\displaystyle{N/2\geq (\alpha \lfloor f(d) \rfloor-1)/2}$ inside $B(0,R(d))$. By a suitable choice of $\alpha\in (0,1)$, we can ensure that this depth is at least $\displaystyle{\frac{1}{30}\log(\epsilon \log(d)/2)}$ for $d$ sufficiently large. Hence the expected number of nests of depth at least $\displaystyle{\frac{1}{30}\log(\epsilon \log(d)/2)}$ inside $B(0,R(d))$ is greater than or equal to $d^{-\epsilon}$. Finally, we can pack
 $\displaystyle{O\left( \frac{1}{R(d)^2}\right)=O\left(\frac{d}{\log(\log(d))}\right)}$ disjoint balls of Fubini-Study radius $\arctan(R(d))$ inside $\mathbb{R}P^2$, and use the linearity of expectation to obtain the result. \hfill $\Box$

This result tells us that one should expect to find nests of depth at least $O\left(\log(\log (d))\right)$ in abundance in a Kostlan random degree $d$ curve. A natural question would be to find upper estimates for expected depths of random nests. Let us define a slightly different random variable: For $p\in \mathbb{R}P^{2}$, let $D_{p,d}$ the {\it expected depth} of the point $p$, which is the integer valued random variable whose value on a Kostlan random curve $C$ of degree $d$ is the depth of the largest nest of $C$ which contains the point $p$ on the interior of its innermost oval. We previously noted the deterministic upper-bound $\displaystyle{ D_{p,d}\leq d/2}$. We prove the following quick-and-easy upper bound for the expectation of $D_{p,d}$. 

\begin{proposition} 
For every $d>0$ and every $p\in \mathbb{R}P^2$,  $\displaystyle{\mathbb{E}(D_{p,d})\leq \frac{\sqrt{d}}{2}\cdot}$
\end{proposition} 

\noindent \textit{Proof:} By unitary invariance of the Kostlan measure, the result is independent of the choice of the point $p\in \mathbb{R}P^2$. Fix a line $\ell$ passing through $p$. The restriction of a Kostlan random polynomial to $\ell$ will be a univariate Kostlan random polynomial. Any oval that contains $p$ on its interior must intersect $\ell$ in at least two points. Hence, $\mathbb{E}(D_{p,d})$ is bounded above by half of the expected number of real roots of a Kostlan random univariate polynomial. It is a well-known result that the latter is equal to $\sqrt{d}$ for every $d$ (see, for instance, \cite{EK}). The proof then follows. \hfill $\Box$

\subsection{Probability of points to remain in different connected components of the complement of a curve}
Let $m$ be a fixed positive integer and $\epsilon>0$. Let $\delta(d):=d^{-\frac{1}{2}+\epsilon}$. Suppose that $p_{1},\ldots,p_{m}$ are $m$ points in $\mathbb{R}P^2$ such that $d_{FS}(p_{i},p_{j})\geq 2\delta(d)$ for all $i\neq j$. Let $\psi_{i}$ be a unitary transformation taking $\mathbb{R}P^2$ to itself and such that $\psi_{i}([1:0:0])=p_{i}$ for $i\in \{1,\ldots,m\}$. We set 
\[ P(X_{0},X_{1},X_{2})=X_{0}^{d-2}\left(X_{1}^2+X_{2}^2-\frac{1}{d}X_{0}^2\right) \] 
and 
\[ P_{2}(X_{0},X_{1},X_{2})=\sum_{i=1}^{m} (P\circ \psi_{i})(X_{0},X_{1},X_{2}). \] 

\begin{lemma}\label{manyestimates}
\begin{enumerate}
\item $\displaystyle{ \|P\|_{2}^2=\frac{10}{d^4}+O\left(\frac{1}{d^5}\right).}$ 
\item For $z\in U\cap \mathbb{R}P^2$ and $i\neq j$, $\displaystyle{|h_{d}(P\circ \psi_{i}(z),P\circ \psi_{j}(z))_{FS}|\leq O\left(\exp\left(-d^{2\epsilon}/2\right)\right).}$ 
\item For $i\neq j$, $|\langle P\circ \psi_{i}, P\circ \psi_{j} \rangle_{2}|\leq O\left(\exp\left(-d^{2\epsilon}/2\right)\right).$
\item $\displaystyle{ \|P_{2}\|_{2}^2=\frac{10 m}{d^4}+O\left(\frac{1}{d^5}\right)}$. 
\end{enumerate}
\end{lemma}

\noindent \textit{Proof:} (1) As in the proof of Lemma \ref{L2estimate}, we note that the monomials $X_{0}^d$, $X_{0}^{d-2}X_{1}^2$ and $X_{0}^{d-2}X_{2}^2$ are mutually orthogonal. Then, 
\begin{eqnarray*} 
\|P\|_{2}^2&=& \frac{1}{d^2}\|X_{0}^d\|^2+ \|X_{0}^{d-2}X_{1}^2\|^2+\|X_{0}^{d-2}X_{2}\|^2\\
&=& \frac{2}{d^2(d+2)(d+1)}+\frac{8}{(d+2)(d+1)d(d-1)} \quad \mathrm{by}\, \mathrm{Lemma} \, \ref{L2norms} \\
&=&\frac{10}{d^4}+O\left(\frac{1}{d^5}\right)\cdot
\end{eqnarray*}

(2) Without loss of generality, we may assume that $i=1, p_{1}=[1:0:0]$ and $\psi_{1}=id$. Since $d_{FS}(p_{1},p_{j})\geq 2\delta(d)$ by assumption, at least one of $d_{FS}(p_{1},z)$ or $d_{FS}(p_{j},z)$ is greater than or equal to $\delta(d)$, by triangle inequality. By unitary invariance of $h_{d}$ we may assume, without loss of generality, that $d_{FS}(p_{1},z)\geq \delta(d)$. Then, 
\begin{eqnarray*} 
|h_{d}(P\circ \psi_{1}(z),P\circ \psi_{j}(z))_{FS}| &\leq& \|P(z)\|_{FS}\|P\circ\psi_{j}(z)\|_{FS} \quad \mathrm{by}\, \mathrm{Cauchy-Schwarz}\,\mathrm{inequality} \\
&=& \frac{ \left| \|z\|^2-\frac{1}{d}\right| }{(1+\|z\|^2)^{d/2}} \frac{ \left| \|\psi_{j}(z)\|^2-\frac{1}{d}\right| }{(1+\|\psi_{j}(z)\|^2)^{d/2}} \quad \mathrm{by} \, \eqref{pointwise} \\
&\leq& \frac{1}{(1+\|z\|^2)^{(d-2)/2}} \\
&\leq& \frac{1}{(1+\delta(d)^2)^{(d-2)/2}} \\
&=& \frac{1}{\left(1+\frac{d^{2\epsilon}}{d}\right)^{(d-2)/2}} \\
&=&O\left(\exp\left(-d^{2\epsilon}/2\right)\right).
\end{eqnarray*}

(3) By \eqref{hermprod}, part (2) of this lemma and triangle inequality, 
\begin{eqnarray*}
|\langle P\circ \psi_{i},P\circ \psi_{j}\rangle_{2}|&\leq&\frac{1}{Vol_{FS}(\mathbb{C}P^{2})}\int_{\mathbb{C}P^{2}} |h_{d}(P\circ\psi_{i}(z),P\circ \psi_{j}(z))_{FS}| dVol_{FS}. \\ 
&\leq& O\left(\exp\left(-d^{2\epsilon}/2\right)\right).
\end{eqnarray*} 

(4) 
\begin{eqnarray*}
\|P_{2}\|_{2}^2 &=& \left\langle \sum_{i=1}^{m} P\circ \psi_{i} , \sum_{i=1}^{m} P\circ \psi_{i} \right\rangle_{2} \\
&=& \sum_{i=1}^{m} \|P\circ\psi_{i}\|_{2}^2 + \sum_{i\neq j} \left\langle P\circ \psi_{i}, P\circ \psi_{j} \right\rangle_{2} \\
&=& m\|P\|_{2}^2 +\sum_{i\neq j} \left\langle P\circ \psi_{i}, P\circ \psi_{j} \right\rangle_{2} \quad \mathrm{by}\, \mathrm{unitary} \, \mathrm{invariance} \\
&=& \frac{10 m}{d^4}+O\left(\frac{1}{d^5}\right) \quad \mathrm{by} \, \mathrm{parts} \, \mathrm{(1)} \, \mathrm{and} \, \mathrm{(3)}.
\end{eqnarray*} 
\hfill $\Box$ 

Let now $\displaystyle{r_{1}=\sqrt{\frac{1}{2d}}}$ and $\displaystyle{r_{2}=\sqrt{\frac{3}{2d}}\cdot}$ For each $i\in\{1,\ldots,m\}$ we define the annuli 
\begin{equation} 
A^{(i)}_{r_{1},r_{2}}= \{ p\in \mathbb{R}P^2 |  \arctan(r_{1})\leq d_{FS}(p,p_{i})\leq \arctan(r_{2})\}.
\end{equation} 
Clearly, these annuli are unitary translates of one another and if one of them is centered at $p_{1}=[1:0:0]$ then in affine coordinates and with respect to Euclidean distance, the condition above reads $r_{1}\leq \|z\| \leq r_{2}$ for it, by Lemma \ref{FSdistance} part (b). Let 
\begin{equation} \label{Kforglobal}
 K=\bigcup_{i=1}^{m} \partial A^{(i)}_{r_{1},r_{2}}. 
 \end{equation}

\begin{lemma} \label{P2sup}
For $z\in K$ and $d$ sufficiently large, $\displaystyle{\|P_{2}(z)\|_{FS}^2\geq \frac{1}{90 d^2}}\cdot $ 
\end{lemma} 

\noindent \textit{Proof:} Recall that $\displaystyle{P_{2}(z)=\sum_{i=1}^{n} (P\circ \psi_{i})(z)}$. Suppose that $z\in \partial A^{(j)}_{r_{1},r_{2}}$ and $i\neq j$. Then, for $d$ sufficiently large, $d_{FS}(p_{i},z)\geq \delta(d)$. This implies that 
\[ \|(P\circ \psi_{i})(z)\|_{FS}=O\left(\exp\left(-d^{2\epsilon}/2\right)\right),\]
as in the proof of Lemma \ref{manyestimates}, part (2). On the other hand, for $i=j$, by unitary invariance, the estimation reduces to the one 
in the proof of Lemma \ref{pointwiseestimate} in the case of $f(d)=1$. The result follows by suitably adjusting the constant so as to accomodate for all the subleading terms. \hfill $\Box$

Let now $\widetilde{P}_{2}=P_{2}/\|P_{2}\|_{2}$ be the $L^2$-normalization of $P_{2}$. The following statement is a direct corollary of Lemma \ref{manyestimates} part (4) and Lemma \ref{P2sup}. 

\begin{corollary} \label{P2lowerbound}
For $z\in K$ and for $d$ sufficiently large, $\displaystyle{ \|\widetilde{P}_{2}(z)\|_{FS}^2\geq \frac{d^2}{1000m}\cdot}$
\hfill $\Box$
\end{corollary} 

\begin{proposition} \label{globalhomotopy}
Let $W_{d}=\langle \widetilde{P}_{2}\rangle^{\perp}$ in $H_d:=H^0(\mathbb{C}P^2,\mathcal{O}_{\mathbb{C}P^2}(d))$. If $Q\in W_{d}$ satisfies $\displaystyle{\|Q\|_{2}^2<\frac{d^2}{1000m}}$ and $a\geq1$, then the homology classes of the zero loci of $a\widetilde{P}_{2}+Q$ and $\widetilde{P}_{2}$ are the same in $H_{1}(A^{(j)}_{r_{1},r_{2}},\mathbb{Z}_{2})$ for every $j\in \{1,\ldots,m\}$. 
\end{proposition} 

\noindent\textit{Proof:} This is a direct consequence of \eqref{Kforglobal}, Proposition \ref{homotopy}  and Corollary \ref{P2lowerbound}. \hfill $\Box$ 

\begin{lemma} \label{nontrivial}
For $d$ sufficiently large, the homology class of $\widetilde{P}_{2}$ is non-trivial in $H_{1}(A^{(j)}_{r_{1},r_{2}},\mathbb{Z}_{2})$ for every $j\in \{1,\ldots,m\}$. 
\end{lemma}

\noindent\textit{Proof:} By unitary invariance, it is enough to prove the statement for $j=1$, furthermore we may assume without loss of generality that $p_{1}=[1:0:0]$ and $\psi_{1}=id$. Now, $\displaystyle{P_{2}(z)=\sum_{i=1}^{n} (P\circ \psi_{i})(z)=P+T}$ where $\displaystyle{T=\sum_{i=2}^{n} P\circ \psi_{i}}$. Then the zero loci of $P$ and $P+T$ have the same homology class in $H_{1}(A^{(1)}_{r_{1},r_{2}},\mathbb{Z}_{2})$ by Proposition \ref{homotopy}, since for $z\in \partial A^{(1)}_{r_{1},r_{2}}$, $\displaystyle{\|P(z)\|_{FS}^2=O\left(\frac{1}{d^2}\right)}$ and $\displaystyle{\|T(z)\|_{FS}^2=O\left(\exp(-d^{2\epsilon}/2)\right)}$ as in the proof of Lemma \ref{P2sup}. Since the former homology class is non-trivial, we deduce that the zero locus of $P_{2}$, hence of $\widetilde{P}_{2}$, represents the non-trivial class in $H_{1}(A^{(1)}_{r_{1},r_{2}},\mathbb{Z}_{2})$. \hfill $\Box$

\begin{theorem} \label{finitelymanypoints}
Let $m$ be a positive integer, $\epsilon>0$, and $\displaystyle{p_{1},\ldots,p_{m}\in \mathbb{R}P^2}$ with $\displaystyle{d_{FS}(p_{i},p_{j})\geq 2d^{-\frac{1}{2}+\epsilon}=:2\delta(d)}$ for all $i\neq j$. Then there exists $\beta(m)>0$ such that the probability of all $p_{i}$ to remain in different connected components of the complement of a Kostlan random real plane curve is at least $\displaystyle{d^{-\beta(m)}}$ for $d$ sufficiently large. 
\end{theorem}

\noindent \textit{Proof:} Let $W_{d}=\langle \widetilde{P}_{2}\rangle^{\perp}$ in $H_d:=H^0(\mathbb{C}P^2,\mathcal{O}_{\mathbb{C}P^2}(d))$, as in Proposition \ref{globalhomotopy}. Let $SW_{d}$ denote the unit sphere in $W_{d}$ and $\mu_{W_{d}}$ the Haar probability measure on it. By Corollary \ref{cor:CP2_real_hyperplane}, there exists a constant $C>0$ such that 
\[ \mu_{W_d}(S)
=O(d^{-1}), \quad \mathrm{where}\, \, S=\Big\{\,Q\in SW_d:\ \|Q\|_{L^\infty(\mathbb{C}P^2)}
> C\sqrt{\log d}\,\Big\}.
\]
In particular we may assume that this probability is at most $\displaystyle{1/2}$ for $d$ sufficiently large. Consider polynomials of the form $a\widetilde{P}_{2}+bQ$, where $|a|\geq 1$ and $Q\in SW_{d}\setminus S$. Then, 
\[\|bQ(z)\|_{FS}^2= |b|^2\|Q(z)\|_{FS}^2\leq |b|^2 \|Q\|_{L^\infty(\mathbb{C}P^2)}^2 \leq |b|^2 C^2\log d. \] 

\noindent By Proposition \ref{globalhomotopy} and Lemma \ref{nontrivial} we deduce that if $\displaystyle{|b|^2< \frac{d^2}{1000m C^2 \log d}}$, then the homology class of the zero locus of $a\widetilde{P}_{2}+bQ$ in $H_{1}(A^{(j)}_{r_{1},r_{2}},\mathbb{Z}_{2})$ is non-trivial for every $j\in \{1,\ldots,m\}$, hence all of $p_{1},\ldots,p_{m}$ lie in different connected components of the complement of this zero locus. It remains to estimate from below the probability of the set of such $a\widetilde{P}_{2}+bQ$. Let us consider $\widehat{P}=(a\widetilde{P}_{2}+bQ)/\|a\widetilde{P}_{2}+bQ\|_{2}$, so that $\widehat{P}\in SH_{d}$. Then, the condition above for $a$ and $b$ becomes 
\[ |\langle\widehat{P},\widetilde{P}_{2}\rangle|^2=\frac{|a|^2}{|a|^2+|b|^2\|Q\|_{2}^2} >  \left(1+\frac{d^2}{1000m C^2 \log d}\right)^{-1}=:\lambda^2. \]

\noindent By Lemma \ref{lem:subspace-spherical} and the probability of $SW_{d}\setminus S$ being at least $\displaystyle{\frac{1}{2}}$,
\begin{eqnarray*}
\mu\big\{ \widehat{P}\in SH_d: |\langle P, A\rangle |>\lambda, Q\in SW_{d}\setminus S \big\}
&\geq& \frac{1}{2}(1-\lambda^2)^{N_d-1} \\
&\geq&\frac{1}{2}\left(1-\frac{1000mC^2 \log d}{d^2}\right)^{d^2}\\
&\geq&\exp(-1001mC^2\log d) \quad \mathrm{for}\, d\, \mathrm{large}\\
&=& d^{-1001mC^2}.
\end{eqnarray*}
Choosing $\beta(m)=1001mC^2$, we conclude the proof. \hfill $\Box$

\end{document}